\documentclass[10pt]{article}
\usepackage{amssymb}
\usepackage{amsmath}
\usepackage{amsfonts}
\usepackage[OT2,T1]{fontenc}

\DeclareSymbolFont{cyrletters}{OT2}{wncyr}{m}{n}
\DeclareMathSymbol{\Sha}{\mathalpha}{cyrletters}{"58}
\newtheorem{theorem}{Theorem}[section]

\newtheorem{conjecture}[theorem]{Conjecture}
\newtheorem{corollary}[theorem]{Corollary}

\newtheorem{definition}[theorem]{Definition}

\newtheorem{proposition}[theorem]{Proposition}
\newtheorem{remark}[theorem]{Remark}

\newenvironment{proof}[1][Proof]{\noindent\textbf{#1.} }{\ \rule{0.5em}{0.5em}}
\newdimen\dummy
\dummy=\oddsidemargin
\input{tcilatex}
\begin{document}

\title{{\LARGE Invariance of the parity conjecture for $p$-Selmer groups of
elliptic curves in a $D_{2p^{n}}$-extension}}
\author{Thomas de La Rochefoucauld}
\maketitle

\begin{abstract}
In section 2, we show a $p$-parity result in a $D_{2p^{n}}$-extension of
number fields $L/K$ ($p\geq 5$) for the twist $1\oplus \eta \oplus \tau $:%
\begin{equation*}
W(E/K,1\oplus \eta \oplus \tau )=(-1)^{\left\langle 1\oplus \eta \oplus \tau
,X_{p}(E/L)\right\rangle },
\end{equation*}%
where $E$ is an elliptic curve over $K,$ $\eta $ and $\tau $ are
respectively the quadratic character and an irreductible representation of
degree $2$ of $Gal(L/K)=D_{2p^{n}},$ and $X_{p}(E/L)$ is the $p$-Selmer
group. The main novelty is that we use a congruence result between $%
\varepsilon _{0}$-factors (due to Deligne) for the determination of local
root numbers in bad cases (places of additive reduction above 2 and 3). We
also give applications to the $p$-parity conjecture (using the machinery of
the Dokchitser brothers).
\end{abstract}

\section{The conjecture of Birch and Swinnerton-Dyer and the parity
conjecture}

Let $K$ be a number field and $E$ an elliptic curve defined over $K$. Denote
by $K_{v}$ the completion of $K$ at a place $v$.

We recall a few definitions:

\begin{definition}
(Tate Module):\newline
The $l$-adic Tate module of $E$ is the inverse limit of the system of
multiplication by $l$ maps $E[l^{n+1}]\longrightarrow E[l^{n}]$, where $E[m]$
denotes the kernel of multiplication by $m$ on $E.$\newline
Set $T_{l}(E)=\lim\limits_{\longleftarrow }E[l^{n}]$, $V_{l}(E)=\mathbb{Q}%
_{l}\otimes _{\mathbb{Z}_{l}}T_{l}(E)$ and:%
\begin{equation*}
\sigma _{E/K_{v},l}^{\prime }:Gal(\overline{K}_{v}/K_{v})\longrightarrow
GL(V_{l}(E)^{\ast }).
\end{equation*}
\end{definition}

\begin{definition}
Fix an embedding, $\iota :\mathbb{Q}_{l}\hookrightarrow \mathbb{C}$; we can
then associate to $\sigma _{E/K_{v},l}^{\prime }$ a complex representation $%
\sigma _{E/K_{v},l,\iota }^{\prime }$ of the Weil-Deligne group (see \cite%
{Roh1} \S 13).\newline
One can show that the isomorphism class of $\sigma _{E/K_{v}}^{\prime
}:=\sigma _{E/K_{v},l,\iota }^{\prime }$ is independent of the choice of $l$
and $\iota $ (see \cite{Roh1} \S 13, \S 14, \S 15).
\end{definition}

Denote by $L(E/K,s)$ the global $L$-function, product of local $L$-functions:%
\begin{equation*}
L(E/K,s)=\tprod\limits_{v\text{ finite}}L(E/K_{v},s)\left( =\tprod\limits_{v%
\text{ finite}}L(\sigma _{E/K_{v}}^{\prime },s)\right) ,
\end{equation*}%
defined for $\func{Re}(s)>\dfrac{3}{2}$ (see \cite{Roh1} \S 17 for the
correspondence between the classical definition of $L(E/K_{v},s)$ and the
one using $\sigma _{E/K_{v}}^{\prime }$) and by 
\begin{equation*}
\Lambda (E/K,s)=A(E/K)^{s/2}L(E/K,s)(2(2\pi )^{-s}\Gamma (s))^{[K:\mathbb{Q}%
]}\text{,}
\end{equation*}%
the "complete" $L$-function.

Recall the following classical conjectures:

\begin{conjecture}
(Birch and Swinnerton-Dyer: BSD):\newline
$ord_{s=1}\Lambda (E/K,s)=rk(E/K).$
\end{conjecture}

\begin{conjecture}
\label{conj eq funct}(Functional equation of $\Lambda :$ FE):\newline
$L(E/K,s)$ has a holomorphic continuation to $\mathbb{C}$ and there is a
number\newline
$W(E/K)=\tprod\limits_{v}W(E/K_{v})\in \{\pm 1\}$ such that:%
\begin{equation*}
\Lambda (E/K,s)=W(E/K)\Lambda (E/K,2-s)
\end{equation*}%
(see \cite{Roh1} \S 12 and \S 19 for the definition of $W(E/K_{v}):=W(\sigma
_{E/K_{v}}^{\prime })$ and \cite{Roh1} \S 21 p.157 for the functional
equation of $\Lambda $).
\end{conjecture}

This conjecture is known in a few cases:

\begin{itemize}
\item For elliptic curves over $\mathbb{Q}$ thanks to modularity results on
elliptic curves due to Wiles, Taylor, Breuil, Diamond and Conrad

\item For elliptic curves over a totally real field $K$, we only know a
meromorphic continuation and the functional equation of $\Lambda $ thanks to
a potential modularity result of Wintenberger (see \cite{Wint}) together
with an argument of Taylor.
\end{itemize}

In general, conjecture \ref{conj eq funct} is not known.

The conjecture of Birch and Swinnerton-Dyer implies the following weaker
conjecture:

\begin{conjecture}
(BSD $\left( \func{mod}2\right) $)\newline
$\func{rk}(E/K)\equiv ord_{s=1}\Lambda (E/K,s)$ $\left( \func{mod}2\right) .$
\end{conjecture}

Combining it with the conjectural functional equation we get:

\begin{conjecture}
\label{conj parity}(Parity conjecture)\newline
$(-1)^{\func{rk}(E/K)}=W(E/K).$
\end{conjecture}

Tim and Vladimir Dokchitser showed that this conjecture is true assuming
that the $6^{\infty }$-part of the Tate-Shafarevich group of $E$ over $%
K(E[2])$ is finite (see \cite{Dok3} Th 7.1 p.20).

\begin{definition}
Selmer group:\newline
Let $X_{p}(E/K):=Hom_{\mathbb{Z}_{p}}(S(E/K,p^{\infty }),\mathbb{Q}_{p}/%
\mathbb{Z}_{p})\otimes _{\mathbb{Z}_{p}}\mathbb{Q}_{p},$\newline
where $S(E/K,p^{\infty }):=\lim\limits_{\underset{n}{\longrightarrow }%
}S(E/K,p^{n})$ is the $p^{\infty }$-Selmer group, sitting in an exact
sequence:%
\begin{equation*}
0\longrightarrow E(K)\otimes \mathbb{Q}_{p}/\mathbb{Z}_{p}\longrightarrow
S(E/K,p^{\infty })\longrightarrow \Sha_{E/K}[p^{\infty }]\longrightarrow 0
\end{equation*}
\end{definition}

If we let $\func{rk}_{p}(E/K):=\dim _{_{\mathbb{Q}_{p}}}X_{p}(E/K)=\func{rk}%
(E/K)+\func{cork}_{\mathbb{Z}_{p}}\Sha_{E/K}[p^{\infty }]$, a more
accessible form of the conjecture \ref{conj parity} is the following:

\begin{conjecture}
\label{conj p-parity}($p$-parity conjecture)\newline
$(-1)^{\func{rk}_{p}(E/K)}=W(E/K)$.
\end{conjecture}

If $L/K$ is a finite Galois extension and $\tau $ is a self-dual $\overline{%
\mathbb{Q}}_{p}$-representation of $Gal(L/K)$ then there is an equivariant
form of the conjecture \ref{conj p-parity}:

\begin{conjecture}
($p$-parity conjecture for (self-dual) twists)\newline
$(-1)^{\left\langle \tau ,X_{p}(E/K)\right\rangle }=W(E/K,\tau )$, where $%
W(E/K,\tau )=\tprod\limits_{v}W(\sigma _{E/K_{v}}^{\prime }\otimes \limfunc{%
Res}\nolimits_{D_{v}}\tau )$\newline
(where $D_{v}\subset Gal(L/K)$ is the decomposition group at $v$).
\end{conjecture}

It is this last conjecture in a particular setting that will interest us for
the rest of the paper.

\section{Invariance of the parity conjecture in a $D_{2p^{n}}$-extension 
\label{section}}

\subsection{Statement of the main theorem and applications to the $p$-parity
conjecture}

Let $K$ be a number field, $E/K$ an elliptic curve and $L/K$ a finite Galois
extension such that $Gal(L/K)\simeq D_{2p^{n}},$ with $p\geq 5$ a prime
number.

$D_{2p^{n}}$ admits the following irreducible representations over $%
\overline{\mathbb{Q}}_{p}$:\newline
$\bullet $ $1$ the trivial representation\newline
$\bullet $ $\eta $ the quadratic character\newline
$\bullet $ $\frac{p^{n}-1}{2}$ irreducible representations of degree 2; they
are of the form,%
\begin{equation*}
I(\chi ):=Ind_{C_{p^{n}}}^{D_{2p^{n}}}(\chi )=I(\chi ^{-1}),
\end{equation*}%
where $\chi $ is a non-trivial character of $C_{p^{n}}$ ($I(1)=1\oplus \eta $
is reducible)$.$\newline
See for example \cite{Ser1} for the description of irreducible
representations of $D_{2p^{n}}$.

\medskip

Let $\tau =I(\chi )$ be such an irreducible representation of degree 2.

\begin{theorem}
\label{theo2}With the notation above and $p\geq 5$, we have the following
equality:%
\begin{equation*}
\frac{W(E/K,\tau )}{W(E/K,1\oplus \eta )}=\dfrac{(-1)^{\left\langle \tau
,X_{p}(E/L)\right\rangle }}{(-1)^{\left\langle 1\oplus \eta
,X_{p}(E/L)\right\rangle }}
\end{equation*}%
In other words, the $p$-parity conjecture for $E/K$ tensored by $1\oplus
\eta \oplus \tau $ holds:%
\begin{equation*}
W(E/K,1\oplus \eta \oplus \tau )=(-1)^{\left\langle 1\oplus \eta \oplus \tau
,X_{p}(E/L)\right\rangle }
\end{equation*}
\end{theorem}

\begin{remark}
\label{cas des freres Dok}The Dokchitser brothers have shown that this
equality holds in two different cases:\newline
$\bullet $ \textit{In the case when }$p$\textit{\ is any prime number but
the elliptic curve }$E/K$\textit{\ has a cyclic decomposition group in all
additive places above }$2$\textit{\ and }$3$ (see\textit{\ }\cite{Dok2}
Th.4.2 (1) p.43).\textit{\smallskip }\newline
$\bullet $ \textit{In the case when }$p\equiv 3\mathit{\ }\left( \func{mod}%
4\right) $\textit{\ (without any additional assumption) using a strong
global }$p$-\textit{parity result over totally real fields }(see\textit{\ }%
\cite{Dok3} Prop.6.12 p.18).\textit{\smallskip }\newline
In particular, the statement of Thm.\ref{theo2} also holds for $p=3.$
Futhermore, this case can be proved without using the "painful calculation"
(of \cite{Dok2} p.30) in the case of additive reduction (see the appendix).%
\textit{\medskip }\newline
\textit{Here we prove the equality for all }$p\geq 5$\textit{\ (without any
additional assumption).}
\end{remark}

\begin{corollary}
\begin{equation*}
\dfrac{W(E/K,I(\chi ))}{(-1)^{\left\langle I(\chi ),X_{p}(E/L)\right\rangle }%
}\text{ does not depend on }\chi :C_{p^{n}}\longrightarrow \mathbb{C}^{\ast
}.
\end{equation*}
\end{corollary}

Theorem \ref{theo2} is equivalent to the fact that Hypothesis 4.1 of \cite%
{Dok2} holds for any elliptic curve and any $p>3$ (using a result of the
Dokchitser brothers it is also true for $p=3$, see Remark \ref{cas des
freres Dok} above). Now using the machinery of the Dokchitser brothers (see
Th.4.3 and Th.4.5 in \cite{Dok2}) we have the following theorems:

\begin{theorem}
Let $K$ be a number field, $p\geq 3,$ and $E/K$ an elliptic curve. Suppose $%
F $ is a $p$-extension of a Galois extension $M/K,$ Galois over $K.\ $If the 
$p $-parity conjecture%
\begin{equation*}
(-1)^{\func{rk}_{p}E/L}=W(E/L)
\end{equation*}%
holds for all subfields $K\subset L\subset M,$ then it holds for all
subfields $K\subset L\subset F.$
\end{theorem}

\begin{theorem}
Let $K$ be a number field, $p\geq 3,$ $E/K$ an elliptic curve and $F/K$ a
Galois extension. Assume that the $p$-Sylow subgroup $P$ of $G=Gal(F/K)$ is
normal and $G/P$ is abelian. If the $p$-parity conjecture holds for $E$ over 
$K$ and its quadratic extensions in $F,$ then it holds for all twists of $E$
by orthogonal representations of $G.$
\end{theorem}

\subsection{Reduction to the case of a $D_{2p}$-extension}

Here we reduce the demonstration of Theorem \ref{theo2} by an induction
argument together with the Galois invariance of root numbers due to Rohrlich
(see \cite{Roh3} Theorem 2), to the following statement:\medskip

\begin{proposition}
It is sufficient to prove Theorem \ref{theo2} in the case when $n=1$ (i.e. $%
Gal(L/K)\simeq D_{2p}$).
\end{proposition}

\begin{proof}
Suppose Theorem \ref{theo2} is true for $n=N-1.$ We will show that theorem
is true for $n=N$.\newline
Consider $L/K$ a finite Galois extension such that $Gal(L/K)\simeq
D_{2p^{N}} $ and $\tau $ an irreducible representation of degree $2$ of $%
D_{2p^{N}}.$\newline
$\bullet $ If $\chi $ is not injective, then the statement is known by the
induction hypothesis.\newline
$\bullet $ If $\chi $ is injective:\newline
Let $\sigma =res(I(\chi )):=res_{D_{2p^{N-1}}}^{D_{2p^{N}}}(I(\chi ))$.%
\newline
Then $\sigma =I(\chi ^{\prime })$, where $\chi ^{\prime }:=\chi _{\left\vert
C_{p^{N-1}}\right. }:C_{p^{N-1}}\rightarrow \overline{\mathbb{Q}}_{p}$ is
injective.\newline
We have: $Ind_{D_{2p^{N-1}}}^{D_{2p^{N}}}(\sigma )=\tbigoplus\limits_{\chi
_{0}}I(\chi _{0})$, where the sum is taken over the $\chi _{0}$ such that $%
\chi _{0\left\vert C_{p^{N-1}}\right. }=\chi _{\left\vert C_{p^{N-1}}\right.
}.$\newline
For each such $\chi _{0}$ there is an element of $Aut(\mathbb{C)}$ sending $%
\chi $ into $\chi _{0}$ and $I(\chi )$ into $I(\chi _{0}).$\newline
By inductivity of root numbers in Galois extension:%
\begin{equation*}
W(E/K,\sigma )=W(E/K,Ind_{D_{2p^{N-1}}}^{D_{2p^{N}}}(\sigma )).
\end{equation*}%
By Galois invariance of root numbers:%
\begin{equation*}
W(E/K,I(\chi ^{\prime }))=W(E/K,I(\chi _{0}))\text{, }\forall \chi _{0}\text{
such that }\chi _{0\left\vert C_{p^{N-1}}\right. }=\chi _{\left\vert
C_{p^{N-1}}\right. }.
\end{equation*}%
So $W(E/K,\sigma )=W(E/K,Ind_{D_{2p^{N-1}}}^{D_{2p^{N}}}(\sigma
))=W(E/K,\tau )^{p}=W(E/K,\tau ).$\newline
On the other hand,\newline
$\left\langle \sigma ,X_{p}(E/L)\right\rangle =\left\langle
Ind_{D_{2p^{N-1}}}^{D_{2p^{N}}}(\sigma ),X_{p}(E/L)\right\rangle
=p.\left\langle \tau ,X_{p}(E/L)\right\rangle $,\newline
because $X_{p}(E/L)$ is a $\mathbb{Q}_{p}$-representation$.$\newline
So $(-1)^{\left\langle 1\oplus \eta \oplus \sigma ,X_{p}(E/L)\right\rangle
}=(-1)^{\left\langle 1\oplus \eta \oplus \tau ,X_{p}(E/L)\right\rangle }.$%
\newline
By the induction hypothesis, $(-1)^{\left\langle 1\oplus \eta \oplus \sigma
,X_{p}(E/L)\right\rangle }=W(E/K,\sigma )$. As a result,\newline
$W(E/K,1\oplus \eta \oplus \tau )=(-1)^{\left\langle 1\oplus \eta \oplus
\tau ,X_{p}(E/L)\right\rangle }.$
\end{proof}

\subsection{The case of a $D_{2p}$-extension}

We first restate Theorem \ref{theo2} in the case of a $D_{2p}$-extension.

Let $K$ be a number field, $E/K$ an elliptic curve and $L/K$ a Galois
extension such that $Gal(L/K)\simeq D_{2p}\simeq C_{p}\rtimes C_{2},$ with $%
p\geq 5$ a prime number.

Recall the irreducible representations of $D_{2p}$ over $\overline{\mathbb{Q}%
}_{p}$:\newline
$\bullet $ $1$ the trivial representation\newline
$\bullet $ $\eta $ the quadratic character\newline
$\bullet $ $I(\chi )$ irreducible representations of degree 2, where $\chi $
is a non-trivial character of $C_{p}$.

\begin{theorem}
\label{theo}With the notation above and $p\geq 5$, we have the following
equality:%
\begin{equation*}
\frac{W(E/K,\tau )}{W(E/K,1\oplus \eta )}=\dfrac{(-1)^{\left\langle \tau
,X_{p}(E/L)\right\rangle }}{(-1)^{\left\langle 1\oplus \eta
,X_{p}(E/L)\right\rangle }}.
\end{equation*}%
In other words, the $p$-parity conjecture for $E/K$ tensored by $1\oplus
\eta \oplus \tau $ holds:%
\begin{equation*}
W(E/K,1\oplus \eta \oplus \tau )=(-1)^{\left\langle 1\oplus \eta \oplus \tau
,X_{p}(E/L)\right\rangle }.
\end{equation*}
\end{theorem}

The proof of Theorem \ref{theo} will occupy the rest of section \ref{section}%
.

\bigskip

We use the following notation:\newline
$\bullet $ $v$ a finite place of $K$\newline
$\bullet $ $K_{v}$ the completion of $K$ at $v$\newline
$\bullet $ $q_{v}=l_{v}^{r}$ the cardinality of the residue field of $K_{v}$%
\newline
$\bullet $ $z$ $\mid v$ a finite place of $L$\newline
$\bullet $ $w\mid v$ a finite place of $L^{H}$ (where $H$ is a subgroup of $%
Gal(L/K)=D_{2p}$)\newline
$\bullet $ $\delta =ord_{v}\left( \text{the minimal discriminant of }%
E/K_{v}\right) $\newline
$\bullet $ $\delta _{H}=ord_{w}\left( \text{the minimal discriminant of }%
E/\left( L^{H}\right) _{w}\right) $\newline
$\bullet $ $e_{H}$ the ramification index of $\left( L^{H}\right) _{w}/K_{v}$%
\newline
$\bullet $ $f_{H}$ the residue degree of $\left( L^{H}\right) _{w}/K_{v}$%
\newline
$\bullet $ $\omega _{E/K_{v}}^{0}=$ a minimal invariant differential of $%
E/K_{v}$\newline
$\bullet $ $C_{w}(E/L^{H})=c_{w}(E/L^{H})\omega (H),$\newline
where $%
\begin{cases}
c_{w}(E/L^{H})=\text{ local Tamagawa factor of }E/L^{H} \\ 
\omega (H)=\left\vert \frac{\omega _{E/K_{v}}^{0}}{\omega _{E/\left(
L^{H}\right) _{w}}^{0}}\right\vert _{\left( L^{H}\right) _{w}}%
\end{cases}%
$

Furthermore, if $l_{v}>3$ then:\newline
\begin{tabular}{ll}
$\left\vert \frac{\omega _{E/K_{v}}^{0}}{\omega _{E/\left( L^{H}\right)
_{w}}^{0}}\right\vert _{\left( L^{H}\right) _{w}}$ & $=q^{\frac{\delta
.e_{H}-\delta _{H}}{12}f_{H}}$ \\ 
& $=q^{\left\lfloor \frac{\delta .e_{H}}{12}\right\rfloor f_{H}}$ (in the
case of potentially good reduction)%
\end{tabular}

\medskip

For $D_{2p}$, there is the following equality: $Ind_{\left\{ 1\right\}
}^{D_{2p}}1-2.Ind_{D_{2}}^{D_{2p}}1-Ind_{C_{p}}^{D_{2p}}1+2.1=0$ of virtual
representations of $G$, this gives the $G$-relation $\Theta :\left\{
1\right\} -2D_{2}-C_{p}+2G$ in the sense of \cite{Dok2} (Def 2.1 p.11).

\medskip

We recall two definitions in our setting (i.e. with $\Theta :\left\{
1\right\} -2D_{2}-C_{p}+2D_{2p}$), for general definitions see \cite{Dok2}.

\begin{definition}
(\cite{Dok2}, Def.2.13 p.14): Let $\rho $ be a self-dual $\mathbb{Q}_{p}[G]$%
-representation.\newline
Pick a $G$-invariant non-degenerate $\mathbb{Q}_{p}$-linear pairing $%
\left\langle ,\right\rangle $ on $\rho $ and set\newline
$C_{\Theta }(\rho )=\det (\left\langle ,\right\rangle \left\vert \rho
^{\left\{ 1\right\} }\right. )\det (\tfrac{1}{2}\left\langle ,\right\rangle
\left\vert \rho ^{D_{2}}\right. )^{-2}\det (\tfrac{1}{p}\left\langle
,\right\rangle \left\vert \rho ^{C_{p}}\right. )^{-1}\det (\tfrac{1}{2p}%
\left\langle ,\right\rangle \left\vert \rho ^{D_{2p}}\right. )^{2}.$\newline
As an element of $\mathbb{Q}_{p}^{\ast }/\mathbb{Q}_{p}^{\ast 2}$, this does
not depend on the choice of the pairing.
\end{definition}

\begin{definition}
(\cite{Dok2}, Def.2.50 p.23) We define:%
\begin{equation*}
T_{\Theta ,p}=\left\{ 
\begin{tabular}{l}
$\sigma $ a self-dual $\overline{\mathbb{Q}}_{p}[G]\text{-}$ \\ 
$\text{representation}$%
\end{tabular}%
\left\vert 
\begin{tabular}{l}
$\left\langle \sigma ,\rho \right\rangle \equiv ord_{p}C_{\Theta }(\rho )$ $%
\left( \func{mod}2\right) $ \\ 
$\forall \rho \text{ a self-dual }\mathbb{Q}_{p}[G]\text{-representation}$%
\end{tabular}%
\right. \right\}
\end{equation*}
\end{definition}

\medskip

Following the approach of the Dokchitser brothers, we have the following
theorem

\begin{theorem}
(Theorem 1.14 of \cite{Dok2}). Let $L/K$ be a Galois extension of number
fields with Galois group $G=D_{2p}$, where $p>2$ is a prime number. Let $%
\Theta :\left\{ 1\right\} -2D_{2}-C_{p}+2D_{2p}$. For every elliptic curve $%
E/K,$ the $\mathbb{Q}_{p}[G]$-represention $X_{p}(E/L)$ is self-dual, and%
\begin{equation*}
\forall \sigma \in T_{\Theta ,p},\quad(-1)^{\left\langle \sigma
,X_{p}(E/L)\right\rangle }=(-1)^{ord_{p}(C)},
\end{equation*}%
where $C=\tprod\limits_{v\nmid \infty }C_{v}$ with $C_{v}=C_{v}(\{1%
\})C_{v}(D_{2})^{-2}C_{v}(C_{p})^{-1}C_{v}(G)^{2}$\newline
and $C_{v}(H)=\tprod\limits_{\underset{w\text{ places of }L^{H}}{w\mid v}%
}C_{w}(E/L^{H}).$
\end{theorem}

Now, since $1\oplus \eta \oplus \tau \in T_{\Theta ,p}$ (see \cite{Dok2},
example 2.53 p.24), we only need to prove that :%
\begin{equation*}
\frac{W(E/K,\tau )}{W(E/K,1\oplus \eta )}=(-1)^{ord_{p}C}. \leqno (1)
\end{equation*}

Furthermore, since we are only interested in the parity of $ord_{p}(C),$ we
do not have to determine $C_{v}(D_{2})$ and $C_{v}(G),$ because these terms
only bring an even contribution (since they appear with an even exponent).

Both sides of $(1)$ are of local nature.

As $W(E/K,\tau )=\tprod\limits_{v}W(E/K_{v},\tau _{v})$, where $\sigma
_{v}:=res_{Gal(L_{z}/K_{v})}\sigma ,$ all we need to do is to prove the
following local equality:%
\begin{equation*}
\frac{W(E/K_{v},\tau _{v})}{W(E/K_{v},(1\oplus \eta )_{v})}%
=(-1)^{ord_{p}\left( C_{v}\right) },\leqno(2)
\end{equation*}%
for each finite place $v$ of $K$ ($v\mid \infty $ do not contribute, since $%
p\neq 2$).

\bigskip

Denote by $G_{v}:=Gal(L_{z}/K_{v})$ the decomposition group of $v$. The
proof of Theorem \ref{theo} split in several cases:\medskip

\begin{tabular}{ll}
${\bullet }$ $G_{v}=\{1\}$ (there are $2p$ places above $v$ in $L$) & see
section \ref{cases 1 et C_p} \\ 
${\bullet }$ $G_{v}=D_{2}$ (there are $p$ places above $v$ in $L$) & see
section \ref{case D_2} \\ 
${\bullet }$ $G_{v}=C_{p}$ (there are $2$ places above $v$ in $L$) & see
section \ref{cases 1 et C_p} \\ 
${\bullet }$ $G_{v}=D_{2p}$ (there is a unique place above $v$ in $L$) & see
section \ref{case D_2p}%
\end{tabular}

\bigskip

We first recall a few facts about the local Tamagawa factors of elliptic
curves.

\subsubsection{Local Tamagawa factors of elliptic curves\label{tamagawa}}

The assumptions and notation from above are in force.

The local Tamagawa factor at $v,$ $c(E/K_{v})=\#\left(
E(K_{v})/E^{0}(K_{v})\right) ,$\newline
$\left( \text{where }E^{0}(K_{v})=\left\{ \text{Points of non-singular
reduction}\right\} \right) $ is determined by Tate's algorithm (see \cite%
{Silv2} IV \S 9):\medskip

$c(E/K_{v})=%
\begin{cases}
1 & 
\begin{array}{c}
\text{if }E\text{ has good reduction at }v%
\end{array}
\\ 
1,\text{ }2,\text{ }3\text{ or }4 & 
\begin{array}{c}
\text{if }E\text{ has additive reduction at }v%
\end{array}
\\ 
n & 
\begin{array}{l}
\text{if }E\text{ has split multiplicative reduction} \\ 
\text{of type }I_{n}\text{ at }v%
\end{array}
\\ 
1\text{ or }2 & 
\begin{array}{l}
\text{if }E\text{ has non-split multiplicative reduction} \\ 
\text{of type }I_{n}\text{ at }v%
\end{array}%
\end{cases}%
$

\medskip

If $E$ acquires semi-stable reduction over $L_{z},$ then:

\begin{enumerate}
\item If $E$ has split multiplicative reduction of type $I_{n}$ over $K_{v},$
then:\newline
$c(E/\left( L^{H}\right) _{w})=n.e_{H}.$

\item If $E$ has non-split multiplicative reduction of type $I_{n}$ over $%
K_{v},$ then:\newline
$c(E/\left( L^{H}\right) _{w})=%
\begin{cases}
n.e_{H} & 
\begin{tabular}{l}
if $E$ has split multiplicative reduction \\ 
$\text{over }\left( L^{H}\right) _{w}$%
\end{tabular}
\\ 
1\text{ or }2 & \text{otherwise.}%
\end{cases}%
$

\item If $E$ has potentially good reduction, then $c(E/\left( L^{H}\right)
_{w})=1,2,3$ or $4.$

\item If $E$ has additive and potentially multiplicative reduction then:%
\newline
$c(E/\left( L^{H}\right) _{w})=%
\begin{cases}
n.e_{H} & 
\begin{array}{l}
\text{if }E\text{ has split multiplicative reduction} \\ 
\text{of type }I_{n}\text{ over }\left( L^{H}\right) _{w}\text{ and }%
l_{v}\neq 2\text{.}%
\end{array}
\\ 
1,2,3\text{ or }4 & 
\begin{array}{c}
\text{otherwise.}%
\end{array}%
\end{cases}%
$
\end{enumerate}

\bigskip

The following few remarks will be used in the subsequent computations.

\begin{remark}
\label{rem1}If $w_{1}$ and $w_{2}$ are two places of $L$ above the same $v,$
then:\newline
$c_{w_{1}}(E/L)=c_{w_{2}}(E/L).$\newline
In particular:%
\begin{equation*}
\begin{cases}
C_{v}(\{1\})=C_{w}(E/L)^{r} \\ 
C_{v}(C_{p})=C_{w^{\prime }}(E/L^{C_{p}})^{r^{\prime }}\text{,}%
\end{cases}%
\end{equation*}%
where $r=$the number of places $w$ of $L$ such that $w\mid v$ and $r^{\prime
}=$the number of places $w^{\prime }$ of $L^{C_{p}}$ such that $w^{\prime
}\mid v$.
\end{remark}

\begin{remark}
\label{rem2}If $E/K$ has potentially good reduction at $v,$ then:\newline
$\forall w$ (resp. $w^{\prime }$) place of $L$ (of $L^{C_{p}}$), $c_{w}(E/L)$
($c_{w^{\prime }}(E/L^{C_{p}})$) $\in \left\{ 1,..,4\right\} $,\newline
and therefore $ord_{p}\left( c_{v}\right) =0$ and $(-1)^{ord_{p}\left(
C_{v}\right) }=(-1)^{ord_{p}\left( \frac{\omega (\left\{ 1\right\} )}{\omega
(C_{p})}\right) }.$
\end{remark}

\begin{remark}
\label{rem3}If the reduction of $E/K$ at $v$ is semi-stable, then $\forall H$
subgroup of $D_{2p},$ $\delta _{H}=\delta .e_{H}$ and therefore $\omega
(H)=1 $ and $(-1)^{ord_{p}\left( C_{v}\right) }=(-1)^{ord_{p}\left(
c_{v}\right) }. $
\end{remark}

\begin{remark}
\label{rem4}If $v\nmid p$ (i.e. $p\neq l_{v},$ $p$ is fixed, $l_{v}$ is
variable), then $ord_{p}\left( \omega (H)\right) =0$ and $%
(-1)^{ord_{p}\left( C_{v}\right) }=(-1)^{ord_{p}\left( c_{v}\right) }$.
\end{remark}

\begin{remark}
\label{rem5}By the previous two remarks, if $E/K$ has good reduction at $v,$
then: $(-1)^{ord_{p}\left( C_{v}\right) }=1.$\newline
As $\frac{W(E/K_{v},\tau _{v})}{W(E/K_{v},(1\oplus \eta )_{v})}=\frac{\det
\tau _{v}(-1)}{\det (1\oplus \eta )_{v}(-1)}=1$ in the case of good
reduction,$\smallskip $\newline
we have the desired equality $(2)$ in the case of good reduction at $v.$
\end{remark}

\begin{remark}
\label{rem6}From \ref{rem2}\ and \ref{rem4} we deduce that the only case
that needs the calculation of both $\omega (H)$ and $c_{w}(E/L^{H})$ is the
case of additive potentially multiplicative reduction at $v\mid p.$
\end{remark}

\medskip

\subsubsection{The cases $G_{v}=\{1\}$ and $G_{v}=C_{p}\label{cases 1 et C_p}
$}

In these cases, $C_{v}(\{1\})$ and $C_{v}(C_{p})$ are squares, so $%
ord_{p}\left( C_{v}\right) \equiv 0\ \left( \func{mod}2\right) .$\newline
$\bullet $ If $G_{v}=\{1\}$, $res_{Gal(L_{z}/K_{v})}\tau =1\oplus 1=(1\oplus
\eta )_{v}$, hence $\frac{W(E/K_{v},\tau _{v})}{W(E/K_{v},(1\oplus \eta
)_{v})}=1.$\newline
$\bullet $ If $G_{v}=C_{p}$, $(1\oplus \eta )_{v}=1\oplus 1$ and $\tau
_{v}=\chi \oplus \chi ^{\ast }$, so%
\begin{equation*}
W(E/K_{v},\tau _{v})=1=W(E/K_{v},(1\oplus \eta )_{v})\text{ (see \cite{Dok2}
lemma A.1 p.47)}.
\end{equation*}%
As a result, in both cases we have: $\frac{W(E/K_{v},\tau _{v})}{%
W(E/K_{v},(1\oplus \eta )_{v})}=1=(-1)^{ord_{p}\left( C_{v}\right) }.$

\subsubsection{The case $G_{v}=D_{2}\label{case D_2}$}

We have $\tau _{v}=(1\oplus \eta )_{v}$, so $\frac{W(E/K_{v},\tau _{v})}{%
W(E/K_{v},(1\oplus \eta )_{v})}=1.$

On the other hand, in this case,\newline
$\forall w^{\prime }\mid v$ place of $L^{C_{p}}$ and $\forall w\mid
w^{\prime }$ place of $L$, 
\begin{tabular}{c}
$L_{w}$ \\ 
$\shortparallel $ \\ 
$\left( L^{C_{p}}\right) _{w^{\prime }}$ \\ 
$\hspace{0.08in}\mid {2}$ \\ 
$K_{v}$%
\end{tabular}

In particular, $C_{v}(\{1\})=C_{v}(C_{p})^{p},$ therefore $%
C_{v}=C_{v}(C_{p})^{p-1}$ and\newline
$ord_{p}\left( C_{v}\right) =0.$

Finally, we get: $\frac{W(E/K_{v},\tau _{v})}{W(E/K_{v},(1\oplus \eta )_{v})}%
=1=(-1)^{ord_{p}\left( C_{v}\right) }.$

\subsubsection{The case $G_{v}=D_{2p}\label{case D_2p}$}

Denote by $w$ (resp $z$) the unique place of $L^{C_{p}}$ (resp $L$) above $v$%
.\newline
In this case, there are two possibilities for the inertia group of $G_{v},$ $%
I_{v}=C_{p}$ or $D_{2p}$ (because $I_{v}$ is a normal subgroup of $%
G_{v}=D_{2p}$ and $G_{v}/I_{v}$ is cyclic).\newline
Furthermore, if $l_{v}\neq p$ then $I_{v}=C_{p}:$

- For $l_{v}\neq 2$ because the inertia group of a tamely ramified extension
is cyclic.

- For $l_{v}=2$ because the case $I_{v}=D_{2p}$, $I_{v}^{wild}=D_{2}$ (the
wild inertia group) is impossible since $I_{v}^{wild}$ is normal in $I_{v}$.

\paragraph{2.3.4.1\hspace{0.2cm}Computation of $(-1)^{ord_{p}\left(
C_{v}\right) }\label{computation ordC_v}$}

\begin{enumerate}
\item \label{cas1 ordC_v}If $E/K_{v}$ has potentially multiplicative
reduction:

\begin{enumerate}
\item If $E/K_{v}$ acquires split multiplicative reduction of type $I_{n}$
over $L_{z}$ (and therefore over $\left( L^{C_{p}}\right) _{w}$), then:%
\newline
$C_{v}(\{1\})=c_{w}(E/L_{z})%
\begin{tabular}[t]{l}
$=e_{L_{z}/\left( L^{C_{p}}\right) _{w}}\times c_{w^{\prime }}(E/\left(
L^{C_{p}}\right) _{w})$ \\ 
$=\dfrac{e_{\{1\}}}{e_{C_{p}}}\times c_{v}(E/K)C_{v}(C_{p})$ \\ 
$=\dfrac{e_{\{1\}}}{e_{C_{p}}}\times C_{v}(C_{p})$%
\end{tabular}%
$\newline
but $%
\begin{cases}
\text{if }I_{v}=C_{p}\text{ then }e_{\{1\}}=p\text{ and }e_{C_{p}}=1 \\ 
\text{if }I_{v}=D_{2p}\text{ then }e_{\{1\}}=2p\text{ and }e_{C_{p}}=2%
\end{cases}%
$\newline
In both cases we get: $C_{v}=p$ and $(-1)^{ord_{p}\left( C_{v}\right) }=-1.$

\item If $E/K_{v}$ does not acquire split multiplicative reduction of type $%
I_{n}$ over $L_{z}$ (and therefore nor over $\left( L^{C_{p}}\right) _{w}$),
then:\newline
$c_{v}(\{1\})$, $c_{v}(C_{p})$ $\in \left\{ 1,2,3,4\right\} $ and $%
ord_{p}\left( \dfrac{\omega \left( \left\{ 1\right\} \right) }{\omega \left(
C_{p}\right) }\right) \equiv 0$ $\left( \func{mod}2\right) .\smallskip $%
\newline
The second claim is a consequence of Remark \ref{rem4} in the case $%
l_{v}\neq p$.\newline
In the case $l_{v}=p,$ we have to distinguish two cases:

\begin{enumerate}
\item If $E/K_{v}$ acquires non-split multiplicative reduction of type $%
I_{n} $ over $L_{z}$ (and therefore over $\left( L^{C_{p}}\right) _{w}$),
then $\delta _{\left\{ 1\right\} }=\delta _{C_{p}}.$\newline
Furthermore, $f_{C_{p}}=f_{\left\{ 1\right\} }=1$ or $2$ and $\dfrac{\omega
\left( \left\{ 1\right\} \right) }{\omega \left( C_{p}\right) }=q^{\delta
f(e_{\left\{ 1\right\} }-e_{C_{p}})},$ so $ord_{p}\left( \dfrac{\omega
\left( \left\{ 1\right\} \right) }{\omega \left( C_{p}\right) }\right)
\equiv 0$ $\left( \func{mod}2\right) $ (because $p-1\mid \left( e_{\left\{
1\right\} }-e_{C_{p}}\right) $).

\item If $E/K_{v},E/\left( L^{C_{p}}\right) _{w}$ and $E/L_{z}$ have
additive reduction (of type $I_{n}^{\ast }$):\newline
$\bullet $ If $I_{v}=C_{p}$, then $f_{C_{p}}=f_{\left\{ 1\right\} }=2$ and
the result follows.\newline
$\bullet $ if $I_{v}=D_{2p}$, since $p\geq 5,$ $E$ becomes of type $%
I_{2n}^{\ast }$ over $\left( L^{C_{p}}\right) _{w}$ and $I_{2pn}^{\ast }$
over $L_{z}$ and we get: $ord_{p}\left( \omega \left( \left\{ 1\right\}
\right) \right) =ord_{p}\left( \omega \left( C_{p}\right) \right) \equiv 0$ $%
\left( \func{mod}2\right) .$
\end{enumerate}
\end{enumerate}

\textit{To sum up, in the case of potentially multiplicative reduction:}%
\newline
$(-1)^{ord_{p}\left( C_{v}\right) }=%
\begin{cases}
-1 & \text{if }E/(L^{C_{p}})\text{ has split multiplicative reduction} \\ 
1 & \text{otherwise.}%
\end{cases}%
$

\item If $E/K_{v}$ has potentially good reduction, then:

\begin{enumerate}
\item If $I_{v}=C_{p}$ (i.e. $e_{\{1\}}=p$ and $e_{C_{p}}=1):$\newline
We get: $f_{\{1\}}=f_{C_{p}}=2$ so $ord_{p}(\omega (C_{p}))\equiv
ord_{p}(\omega (\{1\}))\equiv 0$ $\left( \func{mod}2\right) \smallskip $%
\newline
and therefore $(-1)^{ord_{p}\left( C_{v}\right) }=1$ (see Remark \ref{rem2}).

\item If $I_{v}=D_{2p}$ (i.e. $e_{\{1\}}=2p$, $e_{C_{p}}=2$ and $l_{v}=p):$%
\newline
We get: $\frac{C_{v}(\{1\})}{C_{v}(C_{p})}=\frac{\omega (\{1\})}{\omega
(C_{p})}=q_{v}^{\left\lfloor \frac{\delta .e_{\left\{ 1\right\} }}{12}%
\right\rfloor -\left\lfloor \frac{\delta .e_{C_{p}}}{12}\right\rfloor
}=q_{v}^{\left\lfloor \frac{\delta .2p}{12}\right\rfloor -\left\lfloor \frac{%
\delta .2}{12}\right\rfloor }.$

\begin{enumerate}
\item If $q_{v}$ is an even power of $p,$ then\newline
$(-1)^{ord_{p}\left( C_{v}\right) }=(-1)^{ord_{p}\left( \frac{\omega (\{1\})%
}{\omega (C_{p})}\right) }=1.$

\item \label{tableau}If $q_{v}$ is an odd power of $p:$\newline
A computation of $\left\lfloor \frac{\delta .2p}{12}\right\rfloor $ and $%
\left\lfloor \frac{\delta .2}{12}\right\rfloor $ depending on $p$ modulo $12$
gives the following table:\smallskip \newline
Table of values of $(-1)^{ord_{p}\left( C_{v}\right) }$ depending on the
Kodaira symbol of the curve (and the value of $\mathfrak{e}=\frac{12}{%
\limfunc{pgcd}(\delta ,12)}$) and $p$ $\func{mod}12$:%
\begin{equation*}
\begin{tabular}{|l|c|c|c|c|}
\hline
$p$ $\func{mod}12$ & $1$ & $5$ & $7$ & $11$ \\ \hline
$II,II^{\ast }$ ($\mathfrak{e}=6$) & 1 & -1 & 1 & -1 \\ \hline
$III,III^{\ast }$ ($\mathfrak{e}=4$) & 1 & 1 & -1 & -1 \\ \hline
$IV,IV^{\ast }$ ($\mathfrak{e}=3$) & 1 & -1 & 1 & -1 \\ \hline
$I_{o}^{\ast }$ ($\mathfrak{e}=2$) & 1 & 1 & 1 & 1 \\ \hline
\end{tabular}%
\end{equation*}%
In relation to the above table it may be useful to recall the following
fact: if the residue characteristic of $K_{v}$ is $>3,$ then we have the
following correspondence between $\mathfrak{e}=\frac{12}{\limfunc{pgcd}%
(\delta ,12)},$ the valuation of the minimal discriminant $\delta $\ and the
Kodaira symbols:%
\begin{equation*}
\begin{tabular}[t]{lll}
$\mathfrak{e}=1$ & $\Leftrightarrow \delta =0$ & $\Leftrightarrow E$ is of
type $I_{0}$ \\ 
$\mathfrak{e}=2$ & $\Leftrightarrow \delta =6$ & $\Leftrightarrow E$ is of
type $I_{0}^{\ast }$ \\ 
$\mathfrak{e}=3$ & $\Leftrightarrow \delta =4$ or $8$ & $\Leftrightarrow E$
is of type $IV$ or $IV^{\ast }$ \\ 
$\mathfrak{e}=4$ & $\Leftrightarrow \delta =3$ or $9$ & $\Leftrightarrow E$
is of type $III$ or $III^{\ast }$ \\ 
$\mathfrak{e}=6$ & $\Leftrightarrow \delta =2$ or $10$ & $\Leftrightarrow E$
is of type $II$ or $II^{\ast }$.%
\end{tabular}%
\end{equation*}
\end{enumerate}
\end{enumerate}
\end{enumerate}

\paragraph{2.3.4.2\hspace{0.2cm}Computation of $\frac{W(E/K_{v},\protect\tau %
_{v})}{W(E/K_{v},\left( 1\oplus \protect\eta \right) _{v})}$}

\begin{enumerate}
\item \label{Cas1}The case of potentially multiplicative reduction:\newline
We have an explicit formula of Rohrlich (see \cite{Roh2} Th.2 (ii) p.329):%
\begin{equation*}
W(E/K_{v},\sigma )=\det \sigma (-1)\chi (-1)^{\dim \sigma
}(-1)^{\left\langle \chi ,\sigma \right\rangle }\text{,}
\end{equation*}%
where $\chi $ is the character of $K_{v}^{\ast }$ associated to the
extension $K_{v}(\sqrt{-c_{6}})$ of $K_{v}$ ($c_{6}$ is the classical
factor, see \cite{Silv1} p.46).\newline
Since $\dim \tau _{v}=\dim 1\oplus \eta =2$, $\det (\tau _{v})=\det (1\oplus
\eta )$ and $\left\langle \chi ,\tau _{v}\right\rangle =0$, we get:%
\begin{equation*}
\frac{W(E/K_{v},\tau _{v})}{W(E/K_{v},\left( 1\oplus \eta \right) _{v})}=%
\frac{(-1)^{\left\langle \chi ,\tau _{v}\right\rangle }}{(-1)^{\left\langle
\chi ,\left( 1\oplus \eta \right) _{v}\right\rangle }}=\frac{1}{%
(-1)^{\left\langle \chi ,\left( 1\oplus \eta \right) _{v}\right\rangle }}%
=(-1)^{\left\langle \chi ,\left( 1\oplus \eta \right) _{v}\right\rangle }%
\text{.}
\end{equation*}%
\medskip

\begin{enumerate}
\item If the reduction of $E/K_{v}$ is split multiplicative (i.e. $\chi =1$):%
\newline
Then $(-1)^{\left\langle \chi ,\left( 1\oplus \eta \right) _{v}\right\rangle
}=-1$.

\item If the reduction of $E/K_{v}$ is non-split multiplicative (i.e. $\chi $
is an unramified quadratic character):

\begin{enumerate}
\item If $E$ acquires split multiplicative reduction over $L_{z}$ (and
therefore over $\left( L^{C_{p}}\right) _{w}$), then $\eta _{v}=\chi $,
hence $(-1)^{\left\langle \chi ,\left( 1\oplus \eta \right)
_{v}\right\rangle }=-1.$

\item If $E$ acquires non-split multiplicative reduction over $L_{z}$ (and
therefore over $\left( L^{C_{p}}\right) _{w}$), then $\eta _{v}\neq \chi $,
hence $(-1)^{\left\langle \chi ,\left( 1\oplus \eta \right)
_{v}\right\rangle }=1.$
\end{enumerate}

\item If the reduction of $E/K_{v}$ is additive (i.e. $\chi $ is a ramified
quadratic character)

\begin{enumerate}
\item If $E$ acquires split multiplicative reduction over $L_{z}$ (and
therefore over $\left( L^{C_{p}}\right) _{w}$), then $\eta _{v}=\chi $,
hence $(-1)^{\left\langle \chi ,\left( 1\oplus \eta \right)
_{v}\right\rangle }=-1.$

\item If $E$ acquires non-split multiplicative reduction over $L_{z}$ (and
therefore over $\left( L^{C_{p}}\right) _{w}$), then $\eta _{v}\neq \chi $,
hence $(-1)^{\left\langle \chi ,\left( 1\oplus \eta \right)
_{v}\right\rangle }=1.$
\end{enumerate}
\end{enumerate}

\textit{To sum up, in the case of potentially multiplicative reduction:}%
\newline
$\frac{W(E/K_{v},\tau _{v})}{W(E/K_{v},\left( 1\oplus \eta \right) _{v})}$%
\begin{tabular}[t]{l}
$=%
\begin{cases}
-1 & \text{if }E/(L^{C_{p}})\text{\ has split multiplicative reduction} \\ 
1 & \text{otherwise.}%
\end{cases}%
$ \\ 
$=(-1)^{ord_{p}\left( C_{v}\right) }$, by \ref{computation ordC_v}.1.\ref%
{cas1 ordC_v}%
\end{tabular}

\item The case of potentially good reduction:\newline
Here we have to distinguish the cases $l_{v}=p$ and $l_{v}\neq p.$

\begin{enumerate}
\item The case $l_{v}=p.$\newline
We have again an explicit formula of Rohrlich, since $p\geq 5$ (see \cite%
{Roh2}, Th.2 (iii) p.329):\newline
We use the following notation:{}\newline
$\bullet $ $q=p^{r}$ the cardinality of the residue field residue degree of $%
K_{v}$\newline
$\bullet $ $\mathfrak{e}=\frac{12}{\limfunc{pgcd}(\delta ,12)}$\newline
$\bullet $ $\epsilon =%
\begin{cases}
1 & \text{if }r\text{ is even or }\mathfrak{e}=1 \\ 
\left( \frac{-1}{p}\right) & \text{if }r\text{ is odd and }\mathfrak{e}=2%
\text{ or }6 \\ 
\left( \frac{-3}{p}\right) & \text{if }r\text{ is odd and }\mathfrak{e}=3 \\ 
\left( \frac{-2}{p}\right) & \text{if }r\text{ is odd and }\mathfrak{e}=4%
\text{.}%
\end{cases}%
$\newline
Then $\forall \sigma $ a self-dual representation of $Gal(\overline{K}%
_{v}/K_{v})$ with finite image:%
\begin{equation*}
W(E/K_{v},\sigma )=%
\begin{cases}
\alpha (\sigma ,\epsilon ) & 
\begin{tabular}{l}
if $q\equiv 1[\mathfrak{e}]$%
\end{tabular}
\\ 
\alpha (\sigma ,\epsilon )(-1)^{\left\langle 1+\eta _{nr}+\hat{\sigma}%
_{e},\sigma \right\rangle } & 
\begin{array}{l}
\text{if }q\equiv -1[\mathfrak{e}] \\ 
\text{and }\mathfrak{e}=3,4,6,%
\end{array}%
\end{cases}%
\end{equation*}%
where $\eta _{nr}$ is the unramified quadratic character, $\hat{\sigma}_{e}$
is an irreductible representation of degree 2 of $D_{2\mathfrak{e}}$ and $%
\alpha (\sigma ,\epsilon ):=(\det \sigma )(-1)\epsilon ^{\dim \sigma }$.%
\newline
Since $\dim \tau _{v}=\dim \left( 1\oplus \eta \right) _{v}=2$ and $\det
\tau _{v}=\det \left( 1\oplus \eta \right) _{v}$,\newline
$\alpha (\left( 1\oplus \eta \right) _{v},\epsilon )=\alpha (\tau
_{v},\epsilon )$ and we get:\newline
\begin{tabular}{ll}
$\frac{W(E/K_{v},\tau _{v})}{W(E/K_{v},\left( 1\oplus \eta \right) _{v})}$ & 
$=%
\begin{cases}
1 & 
\begin{tabular}{l}
if $q\equiv 1[\mathfrak{e}]$%
\end{tabular}
\\ 
(-1)^{\left\langle 1+\eta _{nr}+\hat{\sigma}_{e},1+\eta _{v}+\tau
_{v}\right\rangle } & 
\begin{array}{l}
\text{if }q\equiv -1[\mathfrak{e}] \\ 
\text{and }\mathfrak{e}=3,4,6,%
\end{array}%
\end{cases}%
\smallskip $ \\ 
& $=%
\begin{cases}
1 & 
\begin{tabular}{l}
if $q\equiv 1[\mathfrak{e}]$%
\end{tabular}
\\ 
(-1)^{\left\langle 1+\eta _{nr},1+\eta _{v}\right\rangle } & 
\begin{array}{l}
\text{if }q\equiv -1[\mathfrak{e}] \\ 
\text{and }\mathfrak{e}=3,4,6,%
\end{array}%
\end{cases}%
$ \\ 
& ($\left\langle \hat{\sigma}_{e},\tau _{v}\right\rangle =0$ since $%
\mathfrak{e}=3,4,6$ and $p\geq 5$).$\bigskip $%
\end{tabular}

\begin{enumerate}
\item If $r$ is even, then $q\equiv 1[\mathfrak{e}]$ $\forall \mathfrak{e}%
\in \left\{ 2,3,4,6\right\} $ and therefore%
\begin{equation*}
\frac{W(E/K_{v},\tau _{v})}{W(E/K_{v},\left( 1\oplus \eta \right) _{v})}%
=1=(-1)^{ord_{p}\left( C_{v}\right) },
\end{equation*}%
by 2.b.i (in section 2.3.4.1).

\item If $r$ is odd, then $q\equiv 1[\mathfrak{e}]\Longleftrightarrow
p\equiv 1[\mathfrak{e}]$ and:%
\begin{equation*}
\frac{W(E/K_{v},\tau _{v})}{W(E/K_{v},\left( 1\oplus \eta \right) _{v})}=%
\begin{cases}
1 & 
\begin{tabular}{l}
if $q\equiv 1[\mathfrak{e}]$%
\end{tabular}
\\ 
(-1)^{\left\langle 1+\eta _{nr},1+\eta _{v}\right\rangle } & 
\begin{array}{l}
\text{if }q\equiv -1[\mathfrak{e}]\text{ and} \\ 
\mathfrak{e}=3,4,6\text{.}%
\end{array}%
\end{cases}%
\end{equation*}%
\bigskip

\begin{enumerate}
\item If $I_{v}=C_{p},$ then $\eta _{nr}=\eta _{v}$ and $\frac{%
W(E/K_{v},\tau _{v})}{W(E/K_{v},\left( 1\oplus \eta \right) _{v})}=1$.

\item If $I_{v}=D_{2p},$ then $\eta _{nr}\neq \eta _{v}$ and:%
\begin{equation*}
\frac{W(E/K_{v},\tau _{v})}{W(E/K_{v},\left( 1\oplus \eta \right) _{v})}=%
\begin{cases}
1 & \text{if }q\equiv 1[\mathfrak{e}] \\ 
-1 & \text{if }q\equiv -1[\mathfrak{e}]\text{ and }\mathfrak{e}=3,4,6\text{.}%
\end{cases}%
\end{equation*}
\end{enumerate}

In both cases, we obtain for the values of $\frac{W(E/K_{v},\tau _{v})}{%
W(E/K_{v},\left( 1\oplus \eta \right) _{v})}$ exactly the same table as for
the values of $(-1)^{ord_{p}\left( C_{v}\right) },$ depending on $p$ modulo $%
12$:\medskip \newline
Table of values of $\frac{W(E/K_{v},\tau _{v})}{W(E/K_{v},\left( 1\oplus
\eta \right) _{v})}$ depending on the Kodaira symbol of the curve (and the
value of $\mathfrak{e}=\frac{12}{\limfunc{pgcd}(\delta ,12)}$) and $p$ $%
\func{mod}12$:%
\begin{equation*}
\begin{tabular}{|l|l|l|l|l|}
\hline
$p$ $\func{mod}12$ & $1$ & $5$ & $7$ & $11$ \\ \hline
$II,II^{\ast }$ ($\mathfrak{e}=6$) & \multicolumn{1}{|c|}{1} & 
\multicolumn{1}{|c|}{-1} & \multicolumn{1}{|c|}{1} & \multicolumn{1}{|c|}{-1}
\\ \hline
$III,III^{\ast }$ ($\mathfrak{e}=4$) & \multicolumn{1}{|c|}{1} & 
\multicolumn{1}{|c|}{1} & \multicolumn{1}{|c|}{-1} & \multicolumn{1}{|c|}{-1}
\\ \hline
$IV,IV^{\ast }$ ($\mathfrak{e}=3$) & \multicolumn{1}{|c|}{1} & 
\multicolumn{1}{|c|}{-1} & \multicolumn{1}{|c|}{1} & \multicolumn{1}{|c|}{-1}
\\ \hline
$I_{o}^{\ast }$ ($\mathfrak{e}=2$) & \multicolumn{1}{|c|}{1} & 
\multicolumn{1}{|c|}{1} & \multicolumn{1}{|c|}{1} & \multicolumn{1}{|c|}{1}
\\ \hline
\end{tabular}%
\end{equation*}
\end{enumerate}

\item The case $l_{v}\neq p:$\newline
In this case, the explicit formula of Rohrlich cannot be used, since $l_{v}$
can be $2$ or $3.$\newline
Let $\sigma $ be a representation $\sigma $ $:$ $Gal(\overline{K}%
_{v}/K_{v})\rightarrow GL(V_{\sigma })$ with finite image; let $\sigma
_{E/K_{v}}^{\prime }:WD(\overline{K}_{v}/K_{v})\rightarrow GL(V)$ be the
representation of the Weil-Deligne group associated to the elliptic curve
given by $\left( \sigma _{E/K_{v}},N\right) =\left( \sigma
_{E/K_{v}},0\right) $ (because the reduction is potentially good), therefore
this is simply a representation of the Weil group $W(\bar{K}_{v}/K_{v})$
(because $N=0$) and 
\begin{equation*}
\sigma _{E/K_{v}}^{\prime }\otimes \sigma =\sigma _{E/K_{v}}\otimes \sigma
:W(\overline{K}_{v}/K_{v})\rightarrow GL(W),
\end{equation*}%
where $W=V\otimes V_{\sigma },$ is also a representation of the Weil group.%
\newline
We first recall the link between $\varepsilon $-factors and root numbers:%
\begin{equation*}
W(E/K_{v},\sigma )=\dfrac{\varepsilon (\sigma _{E/K_{v}}\otimes \sigma ,\psi
,dx)}{\left\vert \varepsilon (\sigma _{E/K_{v}}\otimes \sigma ,\psi
,dx)\right\vert }=\varepsilon (\sigma _{E/K_{v}}^{\prime }\otimes \sigma
,\psi ,dx_{\psi }),
\end{equation*}%
where $dx$ is any Haar measure, $\psi $ is any additive character of $K_{v}$
and $dx_{\psi }$ the self-dual Haar measure with respect to $\psi $ on $%
K_{v} $.\newline
Here, we choose an additive character $\psi $ for which the Haar measure $%
dx_{\psi }$ takes values (on open compact subsets of $K_{v}$) in $\mathbb{Z}%
_{p}[\zeta _{p}],$ where $\zeta _{p}$ is a primitive $p$-th root of unity.
For example, if the conductor of $\psi $ is trivial, then the values of $%
dx_{\psi }$ lie in $l_{v}^{\mathbb{Z}}\cup \{0\}\subset \mathbb{Z}_{p}[\zeta
_{p}]$.\newline
In one of his articles (\cite{Del} p.548), Deligne gives a description of
the $\varepsilon $-factors in terms of $\varepsilon _{0}$-factors; in our
settings this gives:%
\begin{equation*}
\varepsilon (\sigma _{E/K_{v}}\otimes \sigma ,\psi ,dx_{\psi })=\varepsilon
_{0}(\sigma _{E/K_{v}}\otimes \sigma ,\psi ,dx_{\psi })\det (-\nu (\phi
)\mid W^{I(v)}),
\end{equation*}%
where $\phi $ is the geometric Frobenius at $v$ and $I(v)=Gal(\bar{K}%
_{v}/K_{v}^{ur}).$\newline
Recall that, since $l_{v}\neq p,$ the inertia group of $D_{2p}$ is $%
I_{v}=C_{p}$.\medskip

\begin{enumerate}
\item If $E$ has additive reduction, denote by $F$ the smallest Galois
extension of $K_{v}^{ur}$ such that $E$ has good reduction over $F$ and set $%
\Phi =Gal(F/K_{v}^{ur})$; then the restiction of $\sigma _{E/K_{v}}$ to $%
I(v) $ factors through $\Phi .$\newline
It is known that:\newline
$\bullet $ For $l_{v}\geq 5,$ $\Phi $ is cyclic of order $\mathfrak{e}=\frac{%
12}{\limfunc{pgcd}(\delta ,12)}$ (dividing 12).\newline
$\bullet $ For $l_{v}=3,$ $\left\vert \Phi \right\vert \in \left\{
2,3,4,6,12\right\} .$\newline
$\bullet $ For $l_{v}=2,$ $\left\vert \Phi \right\vert \in \left\{
2,3,4,6,8,24\right\} .$\newline
For a more precise description of $\Phi $, see, for example, \cite{Bil} or 
\cite{Kraus}.\newline
The representation $\sigma _{E/K}\otimes \sigma $ ($\sigma =\tau _{v}$ or $%
\left( 1\oplus \eta \right) _{v}$) restricted to $I(v)$ factors through a
quotient $H$ of $I(v)$ which admits $\Phi $ and $C_{p}$ as quotients.\newline
We have:%
\begin{equation*}
\left( V\otimes V_{\sigma }\right) ^{I(v)}=\left( V\otimes V_{\sigma
}\right) ^{H}=Hom_{H}(V^{\ast },V_{\sigma })=Hom((V^{\Phi })^{\ast
},V_{\sigma }^{C_{p}})
\end{equation*}%
because $H$ acts on $V$ (resp. on $V_{\sigma }$) through its quotient $\Phi $
(resp. $C_{p}$) and $\left\vert \Phi \right\vert $ is prime to $p$.\newline
Futhermore, $V^{H}=V^{\Phi }=\left\{ 0\right\} $ since $E$ has additive
reduction, hence%
\begin{equation*}
\left( V\otimes V_{\sigma }\right) ^{I(v)}=0,\hspace{0.3cm}\det \left(
-\left( \sigma _{E/K_{v}}^{\prime }\otimes \sigma \right) (\phi )\mid \left(
V\otimes V_{\sigma }\right) ^{I(v)}\right) =1
\end{equation*}%
and%
\begin{equation*}
W(E/K_{v},\sigma )=\varepsilon _{0}(\sigma _{E/K_{v}}\otimes \sigma ,\psi
,dx_{\psi })\qquad (\sigma =\tau _{v},(1\oplus \eta )_{v}).\leqno(3)
\end{equation*}%
\bigskip Deligne also gives congruence results for these $\varepsilon _{0}$ (%
\cite{Del} p.556-557). Since $\chi \equiv 1$ $\func{mod}(1-\zeta _{p})$, we
deduce $I(\chi )\equiv I(1)$ $\func{mod}(1-\zeta _{p})$ and $\sigma
_{E/K_{v}}^{\prime }\otimes \tau _{v}\equiv \sigma _{E/K_{v}}^{\prime
}\otimes \left( 1\oplus \eta \right) _{v}$ $\func{mod}(1-\zeta _{p})$. So
according to Deligne, $\varepsilon _{0}(\sigma _{E/K_{v}}^{\prime }\otimes
\tau _{v},\psi ,dx_{\psi })$ and $\varepsilon _{0}(\sigma _{E/K_{v}}^{\prime
}\otimes \left( 1\oplus \eta \right) _{v},\psi ,dx_{\psi })$ are two
elements of $\{\pm 1\}$ (by (3)), which are congruent modulo $(1-\zeta _{p})$%
, hence the are equal. As a result, 
\begin{equation*}
\frac{W(E/K_{v},\tau _{v})}{W(E/K_{v},\left( 1\oplus \eta \right) _{v})}=1.
\end{equation*}

\item If $E$ has good reduction, then $\sigma _{E/K_{v}}$ is unramified.%
\newline
Then we have:\newline
$\varepsilon (\sigma _{E/K_{v}}\otimes \tau _{v},\psi ,dx)=\varepsilon (\tau
_{v},\psi ,dx)^{\dim \sigma _{E/K_{v}}}\det \sigma _{E/K_{v}}(\Phi ^{m(\tau
_{v},\psi )}),$ where $m(\tau _{v},\psi )\in \mathbb{N}$ depends on
conductors of both $\tau _{v}$ and $\psi $, and the dimension of $\tau _{v}$
(see \cite{Tate1} 3.4.6 p.15), therefore:%
\begin{equation*}
W(E/K_{v},\tau _{v})=W(\sigma _{E/K_{v}}\otimes \tau _{v})=\dfrac{%
\varepsilon (\sigma _{E/K_{v}}\otimes \tau _{v},\psi ,dx)}{\left\vert
\varepsilon (\sigma _{E/K_{v}}\otimes \tau _{v},\psi ,dx)\right\vert }=1%
\text{,}
\end{equation*}%
since $\det \sigma _{E/K_{v}}=1$, $W(\tau _{v})=\dfrac{\varepsilon (\tau
_{v},\psi ,dx)}{\left\vert \varepsilon (\tau _{v},\psi ,dx)\right\vert }=\pm
1$ (because $\det \tau _{v}=1,$ see Proposition p.145 \cite{Roh1}) and $\dim
\sigma _{E/K_{v}}=2.$\newline
Similarly, $W(E/K_{v},\left( 1\oplus \eta \right) _{v})=1,$ so $\frac{%
W(E/K_{v},\tau _{v})}{W(E/K_{v},\left( 1\oplus \eta \right) _{v})}=1.$
\end{enumerate}

In both cases i) and ii) we also have $(-1)^{ord_{p}(C_{v})}=1$ by 2.a. (in
section 2.3.4.1).
\end{enumerate}
\end{enumerate}

\bigskip

To sum up, we have, for each finite prime $v$ of $K$,%
\begin{equation*}
\frac{W(E/K_{v},\tau _{v})}{W(E/K_{v},\left( 1\oplus \eta \right) _{v})}%
=(-1)^{ord_{p}\left( C_{v}\right) }.
\end{equation*}%
This completes the proof of Theorem \ref{theo}.\hspace{5cm}$\blacksquare $

\begin{remark}
This proof can be adjusted to work in the case $Gal(L/K)\simeq D_{2p^{n}},$
the computations are almost the same. The idea to reduce the proof to the
case of a $D_{2p}$-extension, using Galois invariance of Rohrlich, was
suggested to me by Tim Dokchitser.
\end{remark}

\section{Appendix}

The purpose of this appendix is to make a small improvement on Theorem 6.7
of \cite{Dok3}. The interest of this improvement is that Proposition 6.12 of 
\cite{Dok3} (which is the same statement as Theorem \ref{theo2} for $p\equiv
3$ $\func{mod}4)$ will no longer rely on the "truly painful case of additive
reduction" anymore (see \cite{Dok2} p.30). In fact, we use the passage to
the global case to avoid all places of additive reduction, not just those
above $2$ and $3$. Since we have proved the result for $p\geq 5$ (Theorem %
\ref{theo2}) without using any global parity results at all, for us this is
of interest essentially in the case $p=3.$\medskip

\medskip

We start by recalling the definition of an elliptic curve being \textit{%
close to} another one:

\begin{proposition}
\label{Prop courbe proche}Let $\mathcal{E}:$ $%
y^{2}+a_{1}xy+a_{3}y=x^{3}+a_{2}x^{2}+a_{4}x+a_{6}$ be an elliptic curve
over a non archimedean local field $\mathcal{K}$ (with valuation $v$ and
residue characteristic $p$) and $\mathcal{F}/\mathcal{K}$ a finite Galois
extension.\newline
There exists $\varepsilon >0$ such that every elliptic curve $\mathcal{E}%
^{\prime }:$ $y^{2}+a_{1}^{\prime }xy+a_{3}^{\prime }y=x^{3}+a_{2}^{\prime
}x^{2}+a_{4}^{\prime }x+a_{6}^{\prime }$ over $\mathcal{K}$ satisfying:%
\begin{equation*}
\forall i\left\vert a_{i}^{\prime }-a_{i}\right\vert _{v}<\varepsilon ,
\end{equation*}%
the curves $\mathcal{E}$ and $\mathcal{E}^{\prime }$ have the same
conductor, valuation of the minimal discriminant, local Tamagawa factors, $%
C(E/\mathcal{F}^{\prime },\frac{dx}{2y+a_{1}x+a_{3}}),$ root numbers and the
Tate module as a $Gal(\mathcal{\bar{K}}/\mathcal{K)}$-module for $l\neq p$,
in all intermediate fields $\mathcal{F}^{\prime }$ of $\mathcal{F}/\mathcal{K%
}$.\newline
We will say that $\mathcal{E}^{\prime }$ is \textit{close to} $\mathcal{E}/%
\mathcal{K}$.
\end{proposition}

\begin{proof}
This is Proposition 3.3 of \cite{Dok3}.\medskip
\end{proof}

We now state the minor improvement of Theorem 6.7 of \cite{Dok3}:

\begin{theorem}
Let $\mathcal{K}$ a local non archimedean field of characteristic $0$ and $%
\mathcal{F}/\mathcal{K}$ a finite Galois extension. Let $F/K$ be a Galois
extension of totally real fields and $v_{0}$ a place of $K$ such that:%
\newline
$\bullet $ $v_{0}$ admits a unique place $\bar{v}_{0}$ of $F$ above it%
\newline
$\bullet $ $K_{v_{0}}\simeq \mathcal{K}$ and $F_{\bar{v}_{0}}\simeq \mathcal{%
F}$.\newline
Such an extension exists (see Lemma 3.1 of \cite{Dok3}).\newline
Let $\mathcal{E}/\mathcal{K}$ be an elliptic curve with additive reduction.%
\newline
Then there exists an elliptic curve $E/K$ such that:\newline
$\bullet $ $E$ has semi-stable reduction for all $w\neq v_{0}$\newline
$\bullet $ $j(E)$ is not an integer (i.e. $j(E)\notin \mathcal{O}_{K})$%
\newline
$\bullet $ $E/K_{v_{0}}$ is \textit{close to} $\mathcal{E}/\mathcal{K}$.
\end{theorem}

\begin{proof}
We first choose an elliptic curve $E/K$ such that $E/K_{v_{0}}$ is \textit{%
close to} $\mathcal{E}/\mathcal{K}$ (this is possible, by Proposition \ref%
{Prop courbe proche}).\newline
Now the goal is to remove all places of additive reduction by changing $E/K$
to an elliptic curve satisfying the three conditions of the theorem.\newline
Let $E:y^{2}+a_{1}xy+a_{3}y=x^{3}+a_{2}x^{2}+a_{4}x+a_{6}$ with $a_{i}\in 
\mathcal{O}_{K}$.\newline
If we want a place not to be of additive reduction we have to impose one of
the two following conditions:\newline
$\bullet $ The valuation $w(\Delta )$ is zero (in this case $w$ is of good
reduction).\newline
$\bullet $ The valuation $w(c_{4})$ is zero (in this case $w$ is of good or
multiplicative reduction depending on $w(\Delta )=0$ or $>0$).\newline
\smallskip Let $v\neq v_{0}$ be a place of $K$ not above $2$.\newline
To get the condition "$j(E)$ is not an integer" it is sufficient to make $v$
a multiplicative place ($v$ is multiplicative $\Leftrightarrow $ $v(j(E))<0$%
).\newline
\textbf{Step 1}:\textbf{\ Make semi-stable all places }$w\neq v_{0}$ \textbf{%
above }$2$\newline
Denote by $v_{2,1},...,v_{2,r}$ these places.\newline
In this case: $\left[ v_{2,i}(a_{1})=0\Rightarrow v_{2,i}(c_{4})=0\text{ (}%
c_{4}=(a_{1}^{2}+4a_{2})^{2}-24a_{1}a_{3}-48a_{4}\text{)}\right] .$\newline
Let $\mathfrak{p}_{0}$ and $\mathfrak{p}_{2,i}$ be the primes ideals
associated to $v_{0}$ and $v_{2,i}.$\newline
By the Chinese remainder theorem, there exists $d_{1}\in \mathcal{O}_{K}$
such that:\newline
$\bullet $ $d_{1}\equiv 0$ mod $\mathfrak{p}_{0}^{n}$ (i.e. $%
v_{0}(d_{1})\geq n).$\newline
$\bullet $ $d_{1}\equiv 1-a_{1}$ mod $\mathfrak{p}_{2,i}$ $\forall i\in
\left\{ 1,..,r\right\} $ (i.e. $v_{2,i}(a_{1}+d_{1})=0$).\newline
$\bullet $ $d_{1}\equiv -a_{1}$ mod $\mathfrak{p}$ ($\mathfrak{p}$
associated to $v\neq v_{0}$).\newline
So, if we let $a_{1}^{\prime }=a_{1}+d_{1}$ for $n$ big enough we get the
curve $y^{2}+a_{1}^{\prime }xy+a_{3}y=x^{3}+a_{2}x^{2}+a_{4}x+a_{6}$ which
is \textit{close to} $\mathcal{E}/\mathcal{K}$, $v_{2,i}(a_{1}^{\prime
})=v_{2,i}(a_{1}+d_{1})=0$ $\forall i\in \left\{ 1,..,r\right\} $ and $%
v(a_{1}^{\prime })>0.$\newline
\textbf{Step 2}: \textbf{Make }$v$\textbf{\ semi-stable}\newline
By the Chinese remainder theorem, there exist $d_{2},d_{3},d_{4}\in \mathcal{%
O}_{K}$ such that:\newline
$\bullet $ 
\begin{tabular}[t]{l}
$d_{2}\equiv 0$ mod $\mathfrak{p}_{0}^{n}$ (i.e. $v_{0}(d_{2})\geq n)$ \\ 
$d_{2}\equiv 1-a_{2}$ mod $\mathfrak{p}$ (so $v(a_{2}+d_{2})=0$).%
\end{tabular}%
\newline
$\bullet $ 
\begin{tabular}[t]{l}
$d_{3}\equiv 0$ mod $\mathfrak{p}_{0}^{n}$ (i.e. $v_{0}(d_{3})\geq n)$ \\ 
$d_{3}\equiv -a_{3}$ mod $\mathfrak{p}$ (so $v(a_{3}+d_{3})>0$).%
\end{tabular}%
\newline
$\bullet $ 
\begin{tabular}[t]{l}
$d_{4}\equiv 0$ mod $\mathfrak{p}_{0}^{n}$ (i.e. $v_{0}(d_{4})\geq n)$ \\ 
$d_{4}\equiv -a_{4}$ mod $\mathfrak{p}$ (so $v(a_{4}+d_{4})>0$).%
\end{tabular}%
\newline
So, if we let $a_{i}^{\prime }=a_{i}+d_{i},$ $i\in \left\{ 2,3,4\right\} $,
for $n$ big enough we get:\newline
$E^{\prime }:y^{2}+a_{1}^{\prime }xy+a_{3}^{\prime }y=x^{3}+a_{2}^{\prime
}x^{2}+a_{4}^{\prime }x+a_{6}$ is \textit{close to} $\mathcal{E}/\mathcal{K}$
(Proposition \ref{Prop courbe proche}).\newline
Futhermore :%
\begin{tabular}[t]{l}
$\bullet $ $c_{4}^{\prime }=(a_{1}^{\prime 2}+4a_{2}^{\prime
})^{2}-24a_{1}^{\prime }a_{3}^{\prime }-48a_{4}^{\prime }$ \\ 
$\bullet $ $v(a_{1}^{\prime })>0$ \\ 
$\bullet $ $v(a_{3}^{\prime })>0$ \\ 
$\bullet $ $v(a_{4}^{\prime })>0$ \\ 
$\bullet $ $v(a_{2}^{\prime })=0$,%
\end{tabular}%
\newline
so $v(c_{4}^{\prime })=0$.\medskip \newline
The curve $E^{\prime }:y^{2}+a_{1}^{\prime }xy+a_{3}^{\prime
}y=x^{3}+a_{2}^{\prime }x^{2}+a_{4}^{\prime }x+a_{6}$ is \textit{close to} $%
\mathcal{E}/\mathcal{K}$, $\forall w\neq v_{0}$ above\textbf{\ }$2$ $%
w(c_{4}^{\prime })>0$, and $v(c_{4}^{\prime })=0.$ Since $c_{4}^{\prime }$
does not depend on $a_{6},$ we can modify $a_{6}$ to allow places $w\neq
v_{0}$ such that $w(c_{4}^{\prime })>0$ to become places of good reduction
(since $c_{4}^{\prime }$ will be unchanged, some places of good reduction
can become of multiplicative reduction but not of additive reduction) and
such that $v$ is of multiplicative reduction ($v(j(E))<0$).\medskip \newline
\textbf{Step 3}. \textbf{Turn additive reduction places into good reduction
ones and make }$v$\textbf{\ multiplicative.}\newline
Let $v_{1},...,v_{r},v_{r+1},...,v_{t}$ be the places where $%
v_{i}(c_{4}^{\prime })>0,$ $v_{i}\neq v_{0}$ ($\neq v$ and not above $2$).%
\newline
Above, $v_{1},...,v_{r}$ are places of good reduction and $v_{r+1},...,v_{t}$
places of additive reduction of the curve $E^{\prime }$ constructed in step
2.\newline
Let $b_{2},$ $b_{4},$ $b_{6},$ $b_{8}$ and $\Delta $ be the following
classical quantities associatied to $E^{\prime }$:\newline
\begin{tabular}[t]{l}
$b_{2}=a_{1}^{\prime 2}+4a_{2}^{\prime }$ \\ 
$b_{4}=2a_{4}^{\prime }+a_{1}^{\prime }a_{3}^{\prime }$ \\ 
$b_{6}=a_{3}^{\prime 2}+4a_{6}$ \\ 
$b_{8}=a_{1}^{\prime 2}a_{6}+4a_{2}^{\prime }a_{6}-a_{1}^{\prime
}a_{3}^{\prime }a_{4}^{\prime }+a_{2}^{\prime }a_{3}^{\prime
2}-a_{4}^{\prime 2}$%
\end{tabular}%
\newline
and $\Delta 
\begin{tabular}[t]{l}
$=-b_{2}^{2}b_{8}-8b_{4}^{3}-27b_{6}^{2}+9b_{2}b_{4}b_{6}$ \\ 
$=\alpha +\beta .a_{6}+16.a_{6}^{2}$, \\ 
where $\alpha =[-b_{2}^{2}(-a_{1}^{\prime }a_{3}^{\prime }a_{4}^{\prime
}+a_{2}^{\prime }a_{3}^{\prime 2}-a_{4}^{\prime
2})-8b_{4}^{3}-27a_{3}^{\prime 4}+9b_{2}b_{4}a_{3}^{\prime 2}]$ \\ 
and $\beta =[-b_{2}^{3}-216a_{3}^{\prime 2}+36b_{2}b_{4}]$%
\end{tabular}%
$\newline
Let $\gamma =\beta +32.a_{6}$; we know that $16$ is invertible mod $%
\mathfrak{p}_{i}$ $\forall i\in \left\{ 1,..,t\right\} $ (because $\mathfrak{%
p}_{i}$ is not above 2).\smallskip \newline
By the Chinese remainder theorem, there exists $c$ such that:\newline
$\bullet $ $c\equiv 0$ mod $\mathfrak{p}_{0}^{n}$ (i.e. $v_{0}(c)\geq n)$%
\newline
$\bullet $ $c\equiv 0$ mod $\mathfrak{p}_{i}$ $\forall i\in \left\{
1,..,r\right\} $ (i.e. $v_{i}(c)>0$)\newline
$\bullet $ $16c\equiv \alpha _{i}-\gamma $ mod $\mathfrak{p}_{i}$ $\forall
i\in \left\{ r+1,..,t\right\} $ (where $\alpha _{i}\neq 0,\gamma $ mod $%
\mathfrak{p}_{i})$\newline
(i.e. $\forall i\in \left\{ r+1,..,t\right\} ,$ $v_{i}(\gamma +16c)=0$ and $%
v_{i}(c)=0$)\newline
$\bullet $ $c\equiv -a_{6}$ mod $\mathfrak{p}$ (i.e. $v(a_{6}^{\prime })>0)$.%
\newline
Finally, if we let $a_{6}^{\prime }=a_{6}+c$ for $n$ big enough, we get:%
\newline
$E^{\prime \prime }:y^{2}+a_{1}^{\prime }xy+a_{3}^{\prime
}y=x^{3}+a_{2}^{\prime }x^{2}+a_{4}^{\prime }x+a_{6}^{\prime }$\newline
and we see that with this choice:\newline
- $v_{1},...,v_{t}$ are all places of good reduction for $E^{\prime \prime
}. $\newline
- $v$ is a place of multiplicative reduction for $E^{\prime \prime }.$%
\newline
This completes the proof.
\end{proof}

\medskip

\textbf{Acknowledgements: }First of all, I would like to thank my advisor
Jan Nekov\'{a}\v{r}, for suggesting to me this topic, for his guidance along
the work and his careful reading of the different versions of this paper. I
would also like to thank Vladimir and Tim Dokchitser, the first one for his
responses to my questions about their articles, the second one for his
advice.


\begin{thebibliography}{99}
\bibitem{Bil} N.Billerey: D\'{e}faut de semi-stabilit\'{e} des courbes
elliptiques, preprint,\newline
http://people.math.jussieu.fr/${\sim }$billerey/.

\bibitem{Del} P.Deligne: Les constantes des \'{e}quations fonctionnelles des
fonctions $L,$ Modular functions of one variable, II (Proc. Internat. Summer
School, Univ. Antwerp, Antwerp, 1972), 501--97. Lecture Notes in Math., Vol.
349, Springer, Berlin, 1973.

\bibitem{Dok1} T. and V.Dokchitser: Self-duality of Selmer groups, Math.
Proc. Cam. Phil. Soc. 146 (2009), 257--267.

\bibitem{Dok2} T. and V.Dokchitser: Regulator constants and the parity
conjecture, Arxiv: 0709.2852v3, to appear in Invent. Math.

\bibitem{Dok3} T. and V.Dokchitser: Roots numbers and parity of ranks of
elliptics curves, Arxiv: 0906.1815v1.

\bibitem{Kraus} A.Kraus: On le d\'{e}faut de semi-stabilit\'{e} des courbes
elliptiques \`{a} r\'{e}duction additive, Manuscripta Math. 69 (1990),
353--380.

\bibitem{Nek1} J.Nekov\'{a}\v{r}: On the parity of ranks of Selmer groups
III, Documenta Math. 12 (2007), 243--274.

\bibitem{Nek2} J.Nekov\'{a}\v{r}: On the parity of ranks of Selmer groups
IV, Compositio Math. 145 (2009), 1351--1359 (with an appendix by J.-P.
Wintenberger).

\bibitem{Roh1} D.Rohrlich: Weil-Deligne group and elliptic curves.in:
Elliptic curves and related topics, 125--157, CRM Proc. Lecture Notes 4,
AMS, Providence, RI, 1994.

\bibitem{Roh2} D.Rohrlich: Galois theory, elliptic curves and roots numbers,
Comp. Math. 100, n$%
{{}^\circ}%
3$ (1996), 311--349.

\bibitem{Roh3} D.Rohrlich: Galois invariance of local root numbers, preprint,%
\newline
http://math.bu.edu/people/rohrlich/.

\bibitem{Ser1} J.P.Serre: Repr\'{e}sentations lin\'{e}aires des groupes
finis, Hermann.

\bibitem{Silv1} J.H.Silverman: The Arithmetic of Elliptic Curves, GTM 106,
Springer-Verlag

\bibitem{Silv2} J.H.Silverman: Advanced topics in the Arithmetic of Elliptic
Curves, GTM 151, Springer-Verlag.

\bibitem{Tate1} J.Tate: Number: Theoretic Background in: Automorphic forms,
representations and $L$-functions, Part 2 (ed A.Borel and W.Casselman),
3--26, Proc. Symp. in Pure Math. 33 (AMS, Providence, RI, 1979).

\bibitem{Wint} J.P.Wintenberger: Potential modularity of elliptic curves
over totally real fields, appendix to \cite{Nek2}.
\end{thebibliography}
\end{document}